\documentclass[11]{siamltex}

\usepackage[pagewise]{lineno} 

\newcommand{\nwc}{\newcommand}
\usepackage[pdftex]{graphicx}
\usepackage{psfrag}
\usepackage{wrapfig}
\usepackage{epsf}
\usepackage{amsmath,amssymb}
\usepackage[font={footnotesize}]{caption} 
\usepackage{url}
\usepackage[mathscr]{euscript}
\usepackage{xcolor}
\usepackage{subcaption}
\usepackage{tikz}
\usetikzlibrary{patterns}
\usepackage{pgfplots}
\tikzset{every picture/.style={line width=0.75pt}} 

\renewcommand{\theequation}{\arabic{equation}}

\newtheorem{thm}{Theorem}[section]

\newtheorem{rema}[thm]{Remark}

\newcommand{\barint}{\hbox{$\int$\kern-0.75\intwidth
\vrule width 0.5\intwidth height 2.4pt depth -2pt\kern0.25\intwidth}}
\newlength\intwidth
\setbox0=\hbox{$\int$}
\intwidth=\wd0

\newcommand\avint{\hbox{\hbox{$\displaystyle \int$}\hbox{\kern-.9em{$-$}}}}

\newcommand\smavint{\hbox{\hbox{$\int$}\hbox{\kern-.75em{$-$}}}}

\nwc{\st}{^{\mbox{\it st}}}

\nwc{\qref}[1]{(\ref{#1})}
\nwc{\veloc}{v}
\nwc{\rhoc}{\beta}
\nwc{\hl}{\hat{L}}

\def\Xint#1{\mathchoice
{\XXint\displaystyle\textstyle{#1}}%
{\XXint\textstyle\scriptstyle{#1}}%
{\XXint\scriptstyle\scriptscriptstyle{#1}}%
{\XXint\scriptscriptstyle\scriptscriptstyle{#1}}%
\!\int}
\def\XXint#1#2#3{{\setbox0=\hbox{$#1{#2#3}{\int}$}
\vcenter{\hbox{$#2#3$}}\kern-.51\wd0}}

\def\dashint{\Xint-}

\nwc{\intRp}{\int_0^\infty}
\nwc{\aint}{\dashint}
\nwc{\aaint}{\dashint}

\newcommand{\cE}{{\cal E}}
\newcommand{\cF}{{\cal F}}

\newcommand{\R}{\mathbb R}

\newcommand{\be}{\begin{eqnarray}}
\newcommand{\ee}{\end{eqnarray}}

\newcommand{\ben}{\begin{eqnarray*}}
\newcommand{\een}{\end{eqnarray*}}

\newcommand\restr[2]{{
		\left.\kern-\nulldelimiterspace 
		#1 
		\littletaller 
		\right|_{#2} 
}}

\newcommand{\littletaller}{\mathchoice{\vphantom{\big|}}{}{}{}}

\title{}
\author{}

\numberwithin{equation}{section}
\usepackage{indentfirst}
\usepackage{stmaryrd}

\begin{document}
	
\title{A Diffuse Domain Approximation with \\ Transmission-Type Boundary Conditions II: \\ Gamma Convergence}
\author{
	Toai Luong
	\thanks{Corresponding author, Department of Mathematics, The University of Tennessee, Knoxville, TN, USA. 
		Email: tluong4@utk.edu.
	} \and 
	Tadele Mengesha
	\thanks{Department of Mathematics, The University of Tennessee, Knoxville, TN, USA. 
		Email: mengesha@utk.edu.
	} \and 
	Steven M. Wise
	\thanks{Department of Mathematics, The University of Tennessee, Knoxville, TN, USA. 
	Email: swise1@utk.edu
	} \and 
	Ming Hei Wong
	\thanks{Department of Mathematics, The University of Tennessee, Knoxville, TN, USA. 
		Email: mwong4@vols.utk.edu.
	}
}

\date{\today}
\maketitle

    \begin{abstract}
Diffuse domain methods (DDMs) have gained significant attention for solving partial differential equations (PDEs) on complex geometries. These methods approximate the domain by replacing sharp boundaries with a diffuse layer of thickness $\varepsilon$, which scales with the minimum grid size. This reformulation extends the problem to a regular domain, incorporating boundary conditions via singular source terms.  
In this work, we analyze the convergence of a DDM approximation problem with transmission-type Neumann boundary conditions. 
We prove that the energy functional of the diffuse domain problem $\Gamma$--converges to the energy functional of the original problem as $\varepsilon \to 0$. 
Additionally, we show that the solution of the diffuse domain problem strongly converges in $H^1(\Omega)$, up to a subsequence, to the solution of the original problem, as $\varepsilon \to 0$.
    \end{abstract}


    \section{Introduction}

This work is a follow-up to \cite{LuongDDM2025-1} in which we applied formal asymptotics to analyze the approximation of solutions of partial differential equations (PDEs) posed in a domain with complex geometries using a diffuse domain approach. 
This paper focuses on the rigorous variational analysis of the approximation process, where in addition to model approximation, we prove convergence of corresponding solutions. 

PDEs posed in domains with complex geometries arise in various applications, including materials science, fluid dynamics, and biology. 
In many practical problems, these domains may have intricate boundaries, evolving interfaces, or irregular shapes that complicate numerical discretization and analysis. 
Traditional numerical approaches often require conformal meshes that accurately capture domain boundaries. 
Constructing such meshes can be computationally expensive and challenging, especially in scenarios where the domain evolves over time or has small-scale geometric features.

To circumvent these difficulties, diffuse domain methods (DDMs) have emerged as versatile approaches. 
These methods (i) embed the original complex domain into a larger, simpler computational domain, like a square or a cube, and (ii) introduce a phase field function to smoothly approximate the characteristic function of the original domain. 
The governing PDEs are then modified with additional penalization terms that enforce consistency between the diffuse domain approximation and the original sharp-interface formulation. By avoiding the need for complex meshing and allowing for efficient numerical implementation, DDMs have become a widely used technique in various applications, such as phase-field modeling, where they support simulations of complex phenomena in fields like biology (e.g., \cite{Camley-DDM2013, Kockelkoren-DDM2003, Fenton-DDM2005, Lowengrub-DDM2014, Ratz-DDM2015, Voigt-DDM2011-bio}), 
fluid dynamics (e.g., \cite{Voigt-DDM2010, Voigt-DDM2011-fluid, Voigt-DDM2012, Voigt-DDM2014,MR4642032,Voigt-DDM2009}), and materials science (e.g., \cite{Thornton-DDM2016, Ratz-DDM2015, Ratz-DDM2016, Thornton-DDM2018}).

In \cite{LuongDDM2025-1} we have studied the asymptotic convergence of the
diffuse domain approximation problem in one-dimensional space. In addition, we have
provided numerical simulations and discussed their outcomes in relation to our analytical
result. In this paper, we prove the $\Gamma$-convergence of the energy functional associated with the diffuse domain approximation and the convergence of corresponding solutions in the strong $H^1(\Omega)$--topology, in any dimension. 
For motivation and background on diffuse domain problems, 
as well as asymptotic convergence analysis and numerical experiments, we refer the readers to \cite{LuongDDM2025-1} and the references therein.

To be precise, we study the following two-sided boundary value problem in an open cuboidal domain $\Omega$: 
Find a function $u_0: \Omega \to \R$ defined as
\begin{align*}
	u_0(x) = \begin{cases}
		u_1(x),  &\text{if } x \in \Omega_1 \subset \Omega, \\
		u_2(x),  &\text{if } x \in \Omega_2 = \Omega \setminus \overline{\Omega_1},
	\end{cases}
\end{align*}
where $u_1:\Omega_1 \to \R$ and $u_2:\Omega_2 \to \R$ satisfy
    \begin{align}
    \label{bvp1} 
-\Delta u_1 + \gamma u_1 &= q, \quad \text{in } \Omega_1,  
    \\
	\label{bvp2} 
-\alpha\Delta u_2 + \beta u_2 &= h, \quad \text{in } \Omega_2, 
    \\
	\label{bvp3} 
 u_1 &= u_2, \quad \text{on } \partial\Omega_1,  
    \\
	\label{bvp4} 
- \nabla (u_1 - \alpha u_2) \cdot \boldsymbol{n}_1 & = \kappa u_1 + g, \quad \text{on } \partial\Omega_1, 
    \\
	\label{bvp5} 
 \alpha \nabla u_2 \cdot {\boldsymbol{n}_2}  & = 0, \quad \text{on } \partial\Omega.
    \end{align}
Here, we assume the following:
    \begin{enumerate}
    \item[(1)] 
$\Omega_1$ is a bounded open subset of $\R^n$ with a compact $C^3$ boundary $\partial\Omega_1$ satisfying $\overline{\Omega_1} \subset \Omega$ and $\partial\Omega_1 \cap \partial\Omega = \varnothing$,  
and $\Omega_2 := \Omega \setminus \overline{\Omega_1}$ (see Figure~\ref{domain});    

    \item[(2)] 
$\boldsymbol{n}_1$ denotes the outward-pointing  unit normal  vector on $\partial\Omega_1$, and $\boldsymbol{n}_2$ denotes the outward-pointing  unit normal vector on $\partial\Omega$.

	\item[(3)] 
$h, q \in L^2(\Omega)$ and $g \in H^1(\Omega)$ are given functions;
	
	\item[(4)] 
$\alpha, \beta, \gamma$ are given positive constants, and $\kappa$ is a given nonnegative constant.

    \end{enumerate}

\begin{figure}[htb!]
    \centering
    \begin{tikzpicture}[x=0.75pt,y=0.75pt,yscale=-1,xscale=1]

\draw  [line width=2.25]  (272,46) -- (423.67,46) -- (423.67,198.17) -- (272,198.17) -- cycle ;
\draw  [fill={rgb, 255:red, 184; green, 178; blue, 178 }  ,fill opacity=1 ] (380,107.33) .. controls (417,84.33) and (403,159.33) .. (348,163.33) .. controls (293,167.33) and (284,82.33) .. (314,96.33) .. controls (344,110.33) and (343,130.33) .. (380,107.33) -- cycle ;
\draw [line width=0.75]    (325.67,103) -- (338.16,81.59) ;
\draw [shift={(339.67,79)}, rotate = 120.26] [fill={rgb, 255:red, 0; green, 0; blue, 0 }  ][line width=0.08]  [draw opacity=0] (6.25,-3) -- (0,0) -- (6.25,3) -- cycle    ;
\draw [line width=0.75]    (424.17,120) -- (448.33,120.22) ;
\draw [shift={(451.33,120.25)}, rotate = 180.53] [fill={rgb, 255:red, 0; green, 0; blue, 0 }  ][line width=0.08]  [draw opacity=0] (6.25,-3) -- (0,0) -- (6.25,3) -- cycle    ;

\draw (337,130) node [anchor=north west][inner sep=0.75pt]    {$\Omega _{1}$};
\draw (290,165) node [anchor=north west][inner sep=0.75pt]    {$\Omega _{2}$};
\draw (342.33,68) node [anchor=north west][inner sep=0.75pt]    {$\boldsymbol{n}_{1}$};
\draw (454.83,115) node [anchor=north west][inner sep=0.75pt]    {$\boldsymbol{n}_{2}$};

\end{tikzpicture}
    \caption{A domain $\Omega_1$ is covered by a larger cuboidal domain $\Omega$. $\Omega_2 := \Omega \setminus \overline{\Omega_1}$.}
	\label{domain}
\end{figure}
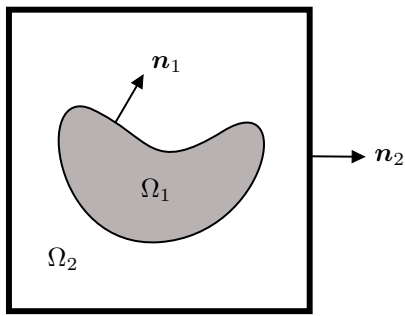

The boundary conditions across the interface $\partial\Omega_1$ are called transmission-type boundary conditions, ensuring continuity of function values across the interface while allowing a jump in flux due to underlying physical mechanisms. 
A solution $u_0$ of the two-sided problem \qref{bvp1}--\qref{bvp5} corresponds to a minimizer of the associated energy functional $\cE_0$, defined by
\begin{align}
    \label{E(u)}
	\cE_0[u] = \int_{\Omega} \left[ \frac{1}{2}(D_0 |\nabla u|^2 + c_0 u^2) - f_0 u \right] dx + \int_{\partial\Omega_1} \left(\frac{1}{2} \kappa u^2 + gu \right) dS, 
\end{align}
for $u \in H^1(\Omega)$, where
\begin{align}
	\label{D-0}D_0(x) &:= \chi_{\Omega_1}(x) + \alpha\chi_{\Omega_2}(x), \\
	\label{c-0}c_0(x) &:= \gamma\chi_{\Omega_1}(x) + \beta\chi_{\Omega_2}(x), \\
	\label{f-0}f_0(x) &:= q(x)\chi_{\Omega_1}(x) + h(x)\chi_{\Omega_2}(x).	
\end{align}
Since $\cE_0$ is coercive and strictly convex, it admits a unique minimizer, ensuring that the two-sided problem \qref{bvp1}--\qref{bvp5} has a unique solution $u_0 \in H^1(\Omega)$.

For each $\varepsilon \in (0,1)$, the diffuse domain approximation of the problem \qref{bvp1}--\qref{bvp5} is given by: 
Find a function $u_\varepsilon : \Omega \to \R$ that satisfies
    \begin{align}
	\label{diff-bvp1} 
-\nabla \cdot (D_\varepsilon \nabla u_\varepsilon) + c_\varepsilon u_\varepsilon + (\kappa u_\varepsilon + g) |\nabla\phi_\varepsilon| &= f_\varepsilon, \quad \text{in } \Omega,  
    \\
	\label{diff-bvp2} 
 D_\varepsilon \nabla u_\varepsilon \cdot \boldsymbol{n}_2 & = 0, \quad \text{on } \partial\Omega,
    \end{align}
where
    \begin{align}
	\label{D-ep}
D_\varepsilon(x) &:= \alpha + (1 - \alpha)\phi_\varepsilon(x) \approx D_0(x), 
    \\
	\label{c-ep}
c_\varepsilon(x) &:= \beta + (\gamma - \beta)\phi_\varepsilon(x) \approx c_0(x), 
    \\
	\label{f-ep}
f_\varepsilon(x) &:= h(x) + [q(x) - h(x)]\phi_\varepsilon(x) \approx f_0(x).	
    \end{align}
Here, $\phi_\varepsilon(x)$ is a phase-field function that approximates the characteristic function $\chi_{\Omega_1}(x)$ of $\Omega_1$. 
A common choice for $\phi_\varepsilon(x)$ is 
\begin{align}\label{phi-ep}
	\phi_\varepsilon(x) = \frac{1}{2}\left[ 1 + \tanh \left( \frac{r(x)}{\varepsilon}\right) \right] \approx \chi_{\Omega_1}(x),
\end{align} 
where $r(x)$ is the signed distance function from $x \in \R^n$ to $\partial\Omega_1$, 
which is assumed to be positive within $\Omega_1$ and negative outside $\overline{\Omega_1}$. 
With this choice of $\phi_\varepsilon(x)$, we note that $|\nabla\phi_\varepsilon(x)|$ approximates the surface delta function $\delta_{\partial\Omega_1}$ of $\partial\Omega_1$.
    
For each $\varepsilon\in (0,1)$, a solution $u_\varepsilon$ of the problem \qref{diff-bvp1}--\qref{diff-bvp2} corresponds to a minimizer of the associated energy functional $\cE_\varepsilon$, defined by
    \begin{align}
    \label{E-eps}
\cE_\varepsilon [u] = \int_\Omega \left[ \frac{1}{2} (D_\varepsilon |\nabla u|^2 + c_\varepsilon u^2) - f_\varepsilon u +  \left(\frac{1}{2} \kappa u^2 + gu \right)  |\nabla\phi_\varepsilon| \right] dx, 
    \end{align}
for $u \in H^1(\Omega)$. 
Since $\cE_\varepsilon$ is coercive and strictly convex, it admits a unique minimizer, ensuring that the diffuse domain problem \qref{diff-bvp1}--\qref{diff-bvp2} has a unique solution $u_\varepsilon \in H^1(\Omega)$.

\section{Main Results}
\subsection{$\Gamma$--convergence of the Energy Functional for the Neumann Boundary Condition case}\label{sec:main-result-2}
In this work, we investigate the sharp interface limit of the energy functional $\cE_\varepsilon$ as $\varepsilon \to 0$, 
using the framework of $\Gamma$--convergence. Specifically, for the Neumann boundary condition case, that is, $\kappa = 0$, we show that the $\Gamma$--limit of the energy $\cE_\varepsilon$ as $\varepsilon \to 0$ with respect to the $L^{2}(\Omega)$--topology is precisely $\cE_0$ in any dimension. In one dimension, the same result holds true for any $\kappa\geq 0$. 

To proceed, we extend the definition of $\cE_\varepsilon$ from \qref{E-eps}, for $\kappa = 0$, to all $u \in L^2(\Omega)$ by defining: 
\begin{align*}
	\cF_\varepsilon [u] = 
	\begin{cases} 
		\displaystyle  \int_\Omega \left[ \frac{1}{2} (D_\varepsilon |\nabla u|^2 + c_\varepsilon u^2) - f_\varepsilon u + gu|\nabla\phi_\varepsilon| \right] dx, &  u \in H^1(\Omega), 
		\\
		\infty, & u \in L^2(\Omega) \setminus H^1(\Omega), 
	\end{cases}
\end{align*}
where $\phi_\varepsilon$, $D_\varepsilon$, $c_\varepsilon$ and $f_\varepsilon$ are defined by \qref{phi-ep}, \qref{D-ep}, \qref{c-ep} and \qref{f-ep}, respectively. 
Similarly, we extend the definition of $\cE_0$ from \qref{E(u)}, for $\kappa = 0$, to all $u \in L^2(\Omega)$ by defining:
\begin{align*}
	\cF_0[u] = 
	\begin{cases} 
		\displaystyle  \int_{\Omega} \left[ \frac{1}{2}(D_0 |\nabla u|^2 + c_0 u^2) - f_0 u \right] dx + \int_{\partial\Omega_1} gu \; dS, & u \in H^1(\Omega), 
		\\
		\infty, & u \in L^2(\Omega) \setminus H^1(\Omega), 
	\end{cases}
\end{align*}
where $D_0, c_0$ and $f_0$ are defined by \qref{D-0}, \qref{c-0} and \qref{f-0}, respectively. 

\begin{theorem}[$\Gamma$--convergence of $\cF_\varepsilon$]\label{thm-main-2}
	As $\varepsilon \to 0$, $\cF_\varepsilon$ $\Gamma$--converges to $\cF_0$ under the strong $L^2(\Omega)$--topology. 
\end{theorem}

\subsection{Strong $H^1(\Omega)$--convergence of the Approximation Solution for the Neumann Boundary Condition case}\label{sec:main-result-3}
We will demonstrate below via a compactness result that the sequence $\{\cF_\varepsilon\}$ is equicoercive. 
As a consequence, in the Neumann boundary condition case ($\kappa = 0$), since $u_0$ is the unique minimizer of $\cF_0$, $u_\varepsilon$ is the unique minimizer of $\cF_\varepsilon$, for each $\varepsilon$, 
    and $\cF_\varepsilon$ $\Gamma$--converges to $\cF_0$ as $\varepsilon \to 0$, 
    it follows from the Fundamental Theorem of $\Gamma$--convergence (see, e.g., \cite{Braides2014}, Theorem~2.1) that  $u_\varepsilon$ converges strongly to $u_0$ in $L^2(\Omega)$ as $\varepsilon \to 0$. While this establishes the convergence of the solution of diffuse domain problem to that of the two-sided problem in $L^{2}(\Omega),$ techniques similar to those used in \cite{Abels-DDM2015} can further show that $u_{\varepsilon}$ converges strongly to $u_0$ in $H^{1}(\Omega)$, up to a subsequence.
    
\begin{theorem}[Strong convergence of $u_\varepsilon$ in $H^1(\Omega)$]\label{thm-main-3}
    Let $u_0$ and $u_\varepsilon$ be the solutions of the two-sided problem \qref{bvp1}--\qref{bvp5} and the diffuse domain approximation problem \qref{diff-bvp1}--\qref{diff-bvp2}, respectively, for $\kappa = 0$. 
	Then, there exists a subsequence of $\{u_\varepsilon\}$ that converges strongly to $u_0$ in $H^1(\Omega)$ as $\varepsilon \to 0$. 
\end{theorem}

\begin{rema}
    The results of Theorem \ref{thm-main-2} and Theorem \ref{thm-main-3} remain valid for general uniformly elliptic quadratic energy functionals. Namely, consider two symmetric, uniformly elliptic and bounded matrices of coefficients,  $\mathbb{A}(x)$ and $\mathbb{B}(x)$. Consider also two functions $\beta(x)$ and $\gamma(x)$ that are bounded from below and above by positive constants.  Define 
\begin{align*}
    \mathbb{D}_\varepsilon (x) := \mathbb{A}(x) + (\mathbb{B}(x) - \mathbb{A}(x)) \phi_{\varepsilon}(x), \quad
    c_{\varepsilon}(x) := \beta(x) + (\gamma(x) - \beta(x))\phi_\varepsilon (x).
\end{align*}
     We introduce the sequence of quadratic energy functionals 
\begin{align*}
	F_\varepsilon [u] = 
	\begin{cases} 
		\displaystyle  \int_\Omega \left[ \frac{1}{2} (\mathbb{D}_\varepsilon \nabla u \cdot \nabla u + c_\varepsilon u^2) - f_\varepsilon u + gu|\nabla\phi_\varepsilon| \right] dx, &  u \in H^1(\Omega), 
		\\
		\infty, & u \in L^2(\Omega) \setminus H^1(\Omega). 
	\end{cases}
\end{align*}
Then, $F_\varepsilon$ $\Gamma$--converges to $F_0$ under the strong $L^2(\Omega)$--topology, as $\varepsilon \to 0$, where 
\begin{align*}
	F_0[u] = 
	\begin{cases} 
		\displaystyle  \int_{\Omega} \left[ \frac{1}{2}(\mathbb{D}_0 \nabla u \cdot \nabla u + c_0 u^2) - f_0 u \right] dx + \int_{\partial\Omega_1} gu \; dS, & u \in H^1(\Omega), 
		\\
		\infty, & u \in L^2(\Omega) \setminus H^1(\Omega).
	\end{cases}
\end{align*}
Here, $\mathbb{D}_0(x)$ and $c_0(x)$ are given by
\begin{align*}
    \mathbb{D}_0 (x) := \mathbb{A}(x) \chi_{\Omega_2}(x) + \mathbb{B}(x)\chi_{\Omega_1}(x), \quad    
    c_0(x) := \beta(x)\chi_{\Omega_2}(x) + \gamma(x)\chi_{\Omega_1}(x).
\end{align*}
Moreover, the sequence of  minimizers $u_{\varepsilon}$ of $F_{\varepsilon}$ converges strongly in $H^{1}(\Omega)$ to the unique minimizer $u_0$ of $F_0$ as $\varepsilon \to 0$, up to a subsequence.  
\end{rema}

\section{Proof of Theorem~\ref{thm-main-2}: $\Gamma$--convergence of $\cF_\varepsilon$}\label{sec:proof-thm-2}
In this section, we prove Theorem~\ref{thm-main-2}. 
To identify the $\Gamma$--limit of a sequence of functionals, we require three essential components (see, e.g., \cite{Braides2002, Leoni-gamma}):
\begin{itemize}
	\item[(i)] A compactness result, which characterizes the limiting functional; 
	
	\item[(ii)] A liminf inequality, which provides a lower bound for the limiting functional; 
	
	\item[(iii)] A recovery sequence satisfying a limsup inequality, ensuring that the lower bound can be achieved.
\end{itemize}

\subsection{Preliminary Lemmas}
For each $\varepsilon \in (0,1)$ and each $w \in H^1(\Omega)$, we define the weighted norms:
\begin{align*}
	\|w\|_{\phi_\varepsilon} := \left[ \int_\Omega \phi_\varepsilon (|\nabla w|^2 + w^2) dx \right]^{1/2},  \quad
	\|w\|_{\delta_\varepsilon} := \left( \int_\Omega w^2 |\nabla\phi_\varepsilon| dx \right)^{1/2}.
\end{align*}
The following Lemmas~\ref{trace-conv}, \ref{trace-ineq} and \ref{diff-conv-bdry}  are results established in \cite{Abels-DDM2015}, corresponding to Theorem 2.3, Lemma 3.5 and Lemma 3.6 in \cite{Abels-DDM2015}, respectively.

\begin{lemma}\label{trace-conv}
	Assume that $\{ w_\varepsilon \} \subset H^1(\Omega)$ satisfies
	\begin{align*}
		 \|w_\varepsilon\|_{\phi_\varepsilon}^2 + \|w_\varepsilon\|_{\delta_\varepsilon}^2 < C,
	\end{align*}
	for some constant $C > 0$ independent of $\varepsilon$. 
	Then, there exist a subsequence of $\{ w_\varepsilon \}$ (not relabeled) and a function $\bar{w} \in H^1(\Omega_1)$ such that $\restr{w_\varepsilon}{\Omega_1} \rightharpoonup \bar{w}$ weakly in $H^1(\Omega_1)$ as $\varepsilon \to 0$, and
	\begin{align*}
		\lim_{\varepsilon \to 0} \int_\Omega f w_\varepsilon |\nabla\phi_\varepsilon| \;dx = \int_{\partial\Omega_1} f \bar{w} \;dS,
	\end{align*}
	for any function $f \in H^1(\Omega)$.
\end{lemma}

\begin{lemma}\label{trace-ineq}
There exists a constant $C_0 > 0$ depending only on $\Omega, \Omega_1$ and the dimension, $n$, 
such that, for any $\varepsilon \in (0,1)$ and any $w \in H^1(\Omega)$, we have
\begin{align*}
	\|w\|_{\delta_\varepsilon} \leq C_0 \|w\|_{H^1(\Omega)}.
\end{align*}
\end{lemma}

\begin{lemma}\label{diff-conv-bdry}
	For any $w \in W^{1,1}(\Omega)$, we have
	\begin{align*}
		\lim_{\varepsilon \to 0} \int_\Omega w |\nabla\phi_\varepsilon| \; dx = \int_{\partial\Omega_1} w \; dS.
	\end{align*}
\end{lemma}

We also prove the following elementary result, demonstrated in Lemma~\ref{diff-conv}.  We recall the definitions of $D_\varepsilon$, $c_\varepsilon$ and $f_\varepsilon$ given in \eqref{D-ep}--\eqref{f-ep}, as well as $D_0, c_0$ and $f_0$ given in \eqref{D-0}--\eqref{f-0}. 

\begin{lemma}\label{diff-conv}
	For any $w \in L^2(\Omega)$, we have
	\begin{align} 
		\label{lim-D-ep} \lim_{\varepsilon \to 0}\int_{\Omega} D_\varepsilon w^2\; dx &= \int_{\Omega} D_0 w^2\; dx, \\ 		
		\label{lim-c-ep} \lim_{\varepsilon \to 0}\int_{\Omega} c_\varepsilon w^2\; dx &= \int_{\Omega} c_0 w^2\; dx, \\
		\label{lim-f-ep} \lim_{\varepsilon \to 0}\int_{\Omega} f_\varepsilon w\; dx &= \int_{\Omega} f_0 w\; dx. 
	\end{align}
\end{lemma}

\begin{proof}	
	Let us prove \qref{lim-D-ep} first. 
	The proof for \qref{lim-c-ep} is similar. 	
	Since $D_\varepsilon w^2 \to D_0 w^2$ a.e. in $\Omega$ as $\varepsilon \to 0$, 
	$|D_\varepsilon  w^2| \leq \max\{ \alpha, 1 \} w^2$, 
	and $w \in L^2(\Omega)$, 
	applying the Dominated Convergence Theorem, we obtain \qref{lim-D-ep}. 
    
    Now we proceed to prove \eqref{lim-f-ep}. 
    Since $|q(\phi_\varepsilon - \chi_{\Omega_1})|^2 \to 0$ a.e. in $\Omega$, $|q(\phi_\varepsilon - \chi_{\Omega_1})|^2 \leq q^2$, and $q\in L^{2}(\Omega)$, applying the Dominated Convergence Theorem, we get
    \begin{align}
    \|q(\phi_\varepsilon - \chi_{\Omega_1})\|_{L^{2}(\Omega)}^2 = \int_\Omega |q(\phi_\varepsilon - \chi_{\Omega_1})|^2 \; dx   \to 0, 
    \end{align}    
    as $\varepsilon \to 0$.  
    Similarly, we also obtain that $\|h((1-\phi_\varepsilon) - \chi_{\Omega_2})\|_{L^{2}(\Omega)} \to 0$ as $\varepsilon \to 0$. 
    Therefore,    
    \begin{align}
    \|f_\varepsilon - f_0\|_{L^2(\Omega)} \leq \|h((1-\phi_\varepsilon) - \chi_{\Omega_2})\|_{L^{2}(\Omega)} +  \|q(\phi_\varepsilon - \chi_{\Omega_1})\|_{L^{2}(\Omega)}\to 0, 
    \end{align}     
    as $\varepsilon \to 0$, 
    which implies that $f_\varepsilon \to f_0$ strongly in $L^2(\Omega)$ as $\varepsilon \to 0$. 
    This convergence of $f_\varepsilon$ and the fact that $w \in L^2(\Omega)$ together imply \eqref{lim-f-ep}.
\end{proof}

\begin{rema}\label{bdd-f-ep}
    We record that $f_\varepsilon \to f_0$ strongly in $L^{2}(\Omega)$ as $\varepsilon\to 0$, and 
    \begin{align}
    \|f_\varepsilon\|_{L^{2}(\Omega)} \leq \|h\|_{L^{2}(\Omega)} + \|q\|_{L^{2}(\Omega)}, \text{ for any } \varepsilon>0.
    \end{align}
    In similar fashion, $c_\varepsilon \to c_0$ and $D_\varepsilon \to D_0$ strongly in $L^{p}(\Omega)$ as $\varepsilon\to 0$, for any $p\in [1, \infty)$. 
\end{rema}

\subsection{Compactness}\label{sec:compactness}
First, we prove the compactness result for $\cF_\varepsilon$.

\begin{theorem}[Compactness]\label{compactness}
	Let $\{ \varepsilon_k \} \subset (0,1)$ be a sequence of numbers such that $\varepsilon_k \searrow 0$ as $k \to \infty$. 
	Let $\{ u_k \} \subset L^2(\Omega)$ be a sequence of functions such that, for any $k = 1,2,\ldots$, 
	$\cF_{\varepsilon_k}[u_k] < M < \infty$, for some $M > 0$ independent of $k$. 
	Then, there exist a subsequence $\{ u_{k_j} \}$ of $\{ u_k \}$ and a function $u \in H^1(\Omega)$ such that
	\begin{align*}
		u_{k_j} &\to u \text{ strongly in } L^2(\Omega), \\
		u_{k_j} &\rightharpoonup u \text{ weakly in } H^1(\Omega), \\
		u_{k_j} &\to u \text{ a.e. in } \Omega,
	\end{align*}
	as $j \to \infty$.	
\end{theorem}	
	
\begin{proof}	
For each $k = 1,2,\ldots$, since $\cF_{\varepsilon_k}[u_k] < M < \infty$, 
by the definition of $\cF_{\varepsilon_k}$, 
we have $u_k \in H^1(\Omega)$ and 
\begin{align*}
	\cF_{\varepsilon_k}[u_k] = \int_\Omega \left[ \frac{1}{2} (D_{\varepsilon_k} |\nabla u_k|^2 + c_{\varepsilon_k} u_k^2) - f_{\varepsilon_k} u_k +  g u_k|\nabla\phi_{\varepsilon_k}| \right] dx.
\end{align*}
Let $\omega := \min \{ \alpha, \beta, \gamma, 1 \} > 0$. 
Since $D_{\varepsilon_k}(x) \geq \min \{ \alpha, 1 \}$ and $c_{\varepsilon_k}(x) \geq \min \{ \beta, \gamma \}$, 
for all $x \in \Omega$ and all $k = 1,2,\ldots$, then
\begin{align}\label{est-0}
	\int_\Omega \left[ \frac{1}{2} (D_{\varepsilon_k} |\nabla u_k|^2 + c_{\varepsilon_k} u_k^2)  \right]\; dx \geq \frac{\omega}{2} \| u_k \|^2_{H^1(\Omega)}.
\end{align}
For any constant $b > 0$, using Young's inequality followed by Remark \ref{bdd-f-ep} we get
\begin{align}\label{est-1}
	\left| \int_\Omega f_{\varepsilon_k} u_k dx \right| &\leq \int_\Omega |f_{\varepsilon_k}| |u_k| dx \nonumber \\
	&\leq \int_\Omega \left( b |u_k|^2 + \frac{1}{4b} |f_{\varepsilon_k}|^2 \right)\; dx \nonumber \\
	&\leq  b \| u_k \|_{H^1(\Omega)}^2 + \frac{1}{2b} \left(\|h\|_{L^2(\Omega)}^2 + \| q\|_{L^2(\Omega)}^2 \right).
\end{align}
By Lemma~\ref{trace-ineq}, there exists a constant $C_0 > 0$, depending only on $\Omega, \Omega_1$ and the dimension $n$, 
such that, for any $k = 1,2,\ldots$ and any $w \in H^1(\Omega)$, we have
\begin{align}
	\|w\|_{\delta_{\varepsilon_k}} \leq C_0 \| w \|_{H^1(\Omega)}.
\end{align}
Then, by Young's inequality, we have, for any $k = 1,2,\ldots$,
\begin{align}\label{est-2}
	\left| \int_\Omega g u_k |\nabla\phi_{\varepsilon_k}|\; dx \right| 
	&\leq \int_\Omega |g| |u_k| |\nabla\phi_{\varepsilon_k}| \;dx  \nonumber \\
	&\leq b \| u_k \|_{\delta_{\varepsilon_k}}^2 + \frac{1}{4b} \| g \|_{\delta_{\varepsilon_k}}^2  \nonumber \\
	&\leq C_0^2 \left( b \| u_k \|_{H^1(\Omega)}^2 + \frac{1}{4b}\| g \|_{H^1(\Omega)}^2 \right).
\end{align} 
Combining \qref{est-0}, \qref{est-1} and \qref{est-2} we get
\begin{align}\label{est-3}
	M > \cF_{\varepsilon_k}[u_k]  
	\geq \left( \frac{\omega}{2} - (C_0^2 + 1)b  \right) \| u_k \|^2_{H^1(\Omega)} - \frac{M_1}{4b}.
\end{align}
for all $k = 1,2,\ldots$, 
where $M_1 := 2\| h \|_{L^2(\Omega)}^2 + 2\| q  \|_{L^2(\Omega)}^2 + \| g \|_{H^1(\Omega)}^2$. 
Choose $b = \omega / (4C_0^2 + 4)$, we obtain
\begin{align}\label{est-4}
	\| u_k \|_{H^1(\Omega)}^2 \leq \frac{4}{\omega} \left( M + \frac{(C_0^2 + 1) M_1}{\omega} \right), 
\end{align}
for all $k = 1,2,\ldots$ 
Thus, using $H^1(\Omega) \subset\subset L^2(\Omega)$ and the weak compactness of $H^1(\Omega)$,
there exist a subsequence $\{ u_{k_j} \}$ of $\{ u_k \}$ and a function $u \in H^1(\Omega)$ such that
\begin{align*}
	u_{k_j} &\to u \text{ strongly in } L^2(\Omega), \\
	u_{k_j} &\rightharpoonup u \text{ weakly in } H^1(\Omega), \\
	u_{k_j} &\to u \text{ a.e. in } \Omega,
\end{align*}
as $j \to \infty$.		
\end{proof}

\begin{rema}The compactness result above establishes the equicoercivity of the sequence $\{\cF_\varepsilon\}$. Indeed, for any sequence $\{u_\varepsilon\}$, if  $\sup_{\varepsilon >0} \cF_{\varepsilon}(u_\varepsilon) < \infty$, the same proof as above shows that $\{u_\varepsilon\}$ is precompact in $L^{2}(\Omega)$. 
\end{rema}

\subsection{Liminf Inequality}\label{sec:liminf-ineq}
Now we prove the following Theorem~\ref{liminf-ineq}, 
which establishes the liminf inequality for the $\Gamma$--convergence result.

\begin{theorem}[Liminf Inequality]\label{liminf-ineq}
	Let $\{ \varepsilon_k \} \subset (0,1)$ be a sequence of numbers such that $\varepsilon_k \searrow 0$ as $k \to \infty$. 
	For any funtion $u \in L^2(\Omega)$ and any sequence $\{ u_k \} \subset L^2(\Omega)$ 
	that satisfies $u_k \to u$ strongly in $L^2(\Omega)$ as $k \to \infty$, we have
	\begin{align}\label{liminf}
		\liminf_{k \to \infty} \cF_{\varepsilon_k}[u_k] \geq \cF_0[u].
	\end{align}
\end{theorem}

    \begin{proof}	
If $\liminf_{k \to \infty} \cF_{\varepsilon_k}[u_k] = \infty$, then \qref{liminf} is trivial. 
Therefore, we only need to consider the case where $\liminf_{k \to \infty} \cF_{\varepsilon_k}[u_k] = L < \infty$.  
Since $u_k \to u$ strongly in $L^2(\Omega)$ as $k \to \infty$, 
there exists a subsquence $\{ k_j \}$ of $\{ k \}$ 
such that $u_{k_j} \to u$ a.e. in $\Omega$ 
and $\cF_{\varepsilon_{k_j}}[u_{k_j}] \to L$ as $j \to \infty$.  
Hence, there exists $j_1 > 0$ such that $\cF_{\varepsilon_{k_j}}[u_{k_j}] < L+1$, for all $j \geq j_1$. 
Consequently, $u_{k_j} \in H^1(\Omega)$, for all $j \geq j_1$. 
Moreover, by the compactness result in Theorem~\ref{compactness}, there exist a further subsequence of $\{ u_{k_j} \}_{j=j_1}^\infty$ (not relabeled)  
and a function $v \in H^1(\Omega)$ such that
\begin{align*}
	u_{k_j} &\to v \text{ strongly in } L^2(\Omega), \\
	u_{k_j} &\rightharpoonup v \text{ weakly in } H^1(\Omega), \\
	u_{k_j} &\to v \text{ a.e. in } \Omega,
\end{align*}
as $j \to \infty$. 
Since $u_{k_j} \to u$ a.e. in $\Omega$ as $j \to \infty$, 
we have  $u = v$ a.e. in $\Omega$, 
which implies that $u \in H^1(\Omega)$ 
and $u_{k_j} \rightharpoonup u$ weakly in $H^1(\Omega)$ as $j \to \infty$.
Therefore, without loss of generality, we assume the following:
\begin{itemize}
	\item[(A1)] $\{ u_k \}_{k=1}^\infty \subset H^1(\Omega)$, which implies that
	\begin{align*}
		\cF_{\varepsilon_k} [u_k] = \int_\Omega \left[ \frac{1}{2} (D_{\varepsilon_k} |\nabla u_k|^2 + c_{\varepsilon_k} u_k^2) - f_{\varepsilon_k} u_k + g u_k |\nabla\phi_{\varepsilon_k}| \right] dx,
	\end{align*}
    for all $k = 1,2,\ldots$

	\item[(A2)] $\lim_{k \to \infty} \cF_{\varepsilon_k}[u_k] = L < \infty$, 
	and $\cF_{\varepsilon_k}[u_k] < L+1$, for all $k = 1,2,\ldots$ 
	
	\item[(A3)] $u \in H^1(\Omega)$, 
	which implies that
	\begin{align*}
		\cF_0[u] = \int_{\Omega} \left[ \frac{1}{2}(D_0 |\nabla u|^2 + c_0 u^2) - f_0 u \right] dx + \int_{\partial\Omega_1} gu\; dS.
	\end{align*}
	
	\item[(A4)] $u_k \to u$ strongly in $L^2(\Omega)$ and a.e. in $\Omega$, 
	and	$u_k \rightharpoonup u$ weakly in $H^1(\Omega)$, as $k \to \infty$. 
\end{itemize}

\

Under Assumptions (A1)--(A4), we will prove the liminf inequality \qref{liminf}. 
Firstly, since $u^2 \in L^1(\Omega)$, by Lemma~\ref{diff-conv}, we have $\int_\Omega c_{\varepsilon_k} u^2 dx \to \int_{\Omega} c_0 u^2 dx$ as $k \to \infty$. 
Hence,
\begin{align}\label{liminf-1a}
	&\left|\int_\Omega c_{\varepsilon_k}u_k^2 - \int_{\Omega} c_0 u^2 \;dx \right|  \nonumber \\
	&\leq \left|\int_\Omega c_{\varepsilon_k} (u_k^2 - u^2)\; dx \right| + \left|\int_\Omega c_{\varepsilon_k}u^2 - \int_{\Omega} c_0 u^2 \;dx \right| \nonumber \\
	&\leq \max \{\beta, \gamma\} \left|\int_\Omega (u_k^2 - u^2) \;dx \right| + \left| \int_\Omega c_{\varepsilon_k}u^2 - \int_{\Omega} c_0 u^2 \;dx \right|
	\to 0,
\end{align}
as $k \to \infty$, 
which implies that
\begin{align}\label{liminf-1}
	\lim_{k \to \infty} \int_\Omega c_{\varepsilon_k} u_k^2 \;dx = \int_{\Omega} c_0 u^2 \;dx.
\end{align}

Secondly, by Lemma~\ref{trace-ineq}, there exists a constant $C_0 > 0$, depending only on $\Omega, \Omega_1$ and the dimension $n$, such that, 
for any $k = 1,2,\ldots$ and any $w \in H^1(\Omega)$, we have
\begin{align}\label{trace-ineq1}
	\|w\|_{\delta_{\varepsilon_k}} \leq C_0 \| w \|_{H^1(\Omega)}.
\end{align}
Since $\cF_{\varepsilon_k}[u_k] < L+1$ for any $k = 1,2,\ldots$, we can apply a similar argument as in the proof of Theorem~\ref{compactness} for \qref{est-4} to obtain
\begin{align}\label{bnd-uk-phi-k}
	\| u_k \|_{H^1(\Omega)}^2 \leq \frac{4}{\omega} \left( L + 1 + \frac{(C_0^2 + 1) M_1}{\omega} \right),
\end{align}
for any $k = 1,2,\ldots$, where $\omega = \min \{ \alpha, \beta, \gamma, 1 \}$ and 
$M_1 = 2\| h \|_{L^2(\Omega)}^2 + 2\| q  \|_{L^2(\Omega)}^2 + \| g \|_{H^1(\Omega)}^2$. 
Combining \qref{trace-ineq1} with \qref{bnd-uk-phi-k}, 
and using the fact that $0 < \phi_{\varepsilon_k} < 1$, 
we obtain
\begin{align}
	\|u_k\|_{\phi_{\varepsilon_k}}^2 + \|u_k\|_{\delta_{\varepsilon_k}}^2 &\leq (1 + C_0^2) \|u_k \|_{H^1(\Omega)}^2 \nonumber \\
	&\leq (1 + C_0^2)\frac{4}{\omega} \left( L + 1 + \frac{(C_0^2 + 1) M_1}{\omega} \right),
\end{align}
for any  $k = 1,2,\ldots$ 
Hence, by Lemma~\ref{trace-conv}, there exist a subsequence of $\{ u_k \}$ (not relabeled) 
and a function $\bar{u} \in H^1(\Omega_1)$ such that 
$\restr{u_k}{\Omega_1} \rightharpoonup \bar{u}$ weakly in $H^1(\Omega_1)$ as $k \to \infty$, and
\begin{align}
	\lim_{k \to \infty} \int_\Omega gu_k  |\nabla\phi_\varepsilon|\; dx = \int_{\partial\Omega_1} g \bar{u} \;dS.
\end{align}
Since $u_k \rightharpoonup u$ weakly in $H^1(\Omega)$ as $k \to \infty$, 
we have $\restr{u}{\Omega_1} = \bar{u}$ a.e. in $\Omega_1$, 
which implies that
\begin{align}\label{liminf-4}
	\lim_{k \to \infty} \int_\Omega g u_k |\nabla\phi_{\varepsilon_k}|\; dx = \int_{\partial\Omega_1} gu\; dS.
\end{align}

Thirdly, using \qref{bnd-uk-phi-k}, we have
\begin{align}
	\int_{\Omega_1} D_{\varepsilon_k} |\nabla u_k|^2\; dx &\leq \max \{\alpha, 1\} \| u_k \|_{H^1(\Omega)}^2 \nonumber \\
	&\leq \max \{\alpha, 1\} \frac{4}{\omega} \left( L + 1 + \frac{(C_0^2 + 1) M_1}{\omega} \right),
\end{align}
for any $k = 1,2,\ldots$ 
Hence, there exist a subsequence of $\{ u_k \}_{k=1}^\infty$ (not relabeled) 
and a function $\Psi \in L^2(\Omega; \R^n)$ such that
\begin{align}\label{sqrt-phik-grad-uk-conv-1}
	\sqrt{D_{\varepsilon_k}}\nabla u_k  \rightharpoonup \Psi \text{ weakly in } L^2(\Omega; \R^n) \text{ as } k \to \infty.
\end{align}
On the other hand, since 
\begin{align}
	\left|\sqrt{D_{\varepsilon_k}} - \sqrt{D_0} \right|^2 \leq 2 (D_{\varepsilon_k} + D_0) 
	\leq 4 \max \{\alpha, 1\},
\end{align}
for any $k = 1,2,\ldots$, 
and $\left|\sqrt{D_{\varepsilon_k}} - \sqrt{D_0} \right|^2 \to 0$ a.e. in $\Omega$ as $k \to \infty$, 
by the Bounded Convergence Theorem, we have
\begin{align}
	\lim_{k \to \infty} \int_{\Omega} \left|\sqrt{D_{\varepsilon_k}} - \sqrt{D_0} \right|^2 \;dx = 0,
\end{align}
which implies that $\sqrt{D_{\varepsilon_k}} \to \sqrt{D_0} $ strongly in $L^2(\Omega)$ as $k \to \infty$.
And since $\nabla u_k \rightharpoonup \nabla u$ weakly in $L^2(\Omega; \R^n)$ as $k \to \infty$, we get
\begin{align}\label{sqrt-phik-grad-uk-conv-2}
	\sqrt{D_{\varepsilon_k}}\nabla u_k \rightharpoonup \sqrt{D_0} \nabla u \text{ weakly in } L^1(\Omega; \R^n) \text{ as } k \to \infty.
\end{align}
Combining \qref{sqrt-phik-grad-uk-conv-1} and \qref{sqrt-phik-grad-uk-conv-2}, we have $\Psi = \sqrt{D_0} \nabla u$ a.e. in $\Omega$, 
which implies that
\begin{align}
	\sqrt{D_{\varepsilon_k}}\nabla u_k \rightharpoonup \sqrt{D_0} \nabla u \text{ weakly in } L^2(\Omega; \R^n) \text{ as } k \to \infty.
\end{align}
Hence, by the lower semicontinuity of norms, we have
\begin{align}\label{liminf-2}
	\int_{\Omega} D_0 |\nabla u|^2 \;dx \leq \liminf_{k \to \infty} \int_{\Omega} D_{\varepsilon_k} |\nabla u_k|^2 \;dx.
\end{align}
Finally, applying Remark \ref{bdd-f-ep}, since $f_{\varepsilon_k} \to f_0$ and $u_{\varepsilon_k} \to u$ strongly in $L^{2}(\Omega)$ as $k \to \infty$, we obtain 
\begin{align}\label{liminf-3}
	\lim_{k \to \infty} \int_\Omega f_{\varepsilon_k} u_k \;dx = \int_{\Omega} f_0 u \;dx.
\end{align}
Combining \qref{liminf-1}, \qref{liminf-4}, \qref{liminf-2} and \qref{liminf-3}, we obtain
\begin{align}
	\liminf_{k \to \infty} \cF_{\varepsilon_k}[u_k] 
	&= \liminf_{k \to \infty} \int_\Omega \left[ \frac{1}{2} (D_{\varepsilon_k} |\nabla u_k|^2 + c_{\varepsilon_k} u_k^2) - f_{\varepsilon_k} u_k + g u_k |\nabla\phi_{\varepsilon_k}| \right] \;dx \nonumber \\
	&\geq \liminf_{k \to \infty} \int_\Omega \frac{1}{2} D_{\varepsilon_k} |\nabla u_k|^2 \;dx + \int_{\Omega} \frac{1}{2} c_0 u^2 \;dx - \int_{\Omega} f_0 u \;dx  + \int_{\partial\Omega_1} gu \;dS \nonumber \\
	&\geq \int_{\Omega} \frac{1}{2} D_0 |\nabla u|^2 \;dx + \int_{\Omega} \frac{1}{2} c_0 u^2\; dx - \int_{\Omega} f_0 u \;dx  + \int_{\partial\Omega_1} gu\; dS \nonumber \\
	&\geq \cF_0[u].
\end{align}
Theorem~\ref{liminf-ineq} is established.
\end{proof}

\subsection{Limsup Inequality and Recovery Sequence}\label{sec:limsup-ineq}
In this section, we prove the following Theorem~\ref{limsup-ineq}, 
which establishes the existence of a recovery sequence.

\begin{theorem}[Limsup Inequality]\label{limsup-ineq}
	Let $\{ \varepsilon_k \} \subset (0,1)$ be a sequence of numbers such that $\varepsilon_k \searrow 0$ as $k \to \infty$. 
	For any funtion $u \in L^2(\Omega)$, there exists a sequence $\{ u_k \} \subset L^2(\Omega)$ such that $u_k \to u$ strongly in $L^2(\Omega)$ as $k \to \infty$, and
	\begin{align*}
		\limsup_{k \to \infty} \cF_{\varepsilon_k}[u_k] \leq \cF_0[u].
	\end{align*}
\end{theorem}

    \begin{proof}	
In the case $u \in L^2(\Omega)$ with $\cF_0[u] = \infty$, any sequence $\{ u_k \} \subset L^2(\Omega)$ that converges strongly to $u$ in $L^2(\Omega)$ can serve as a recovery sequence. 
For simplicity, we choose $u_k = u$, for all $k = 1,2,\ldots$, 
and it is trivial to get 
$\limsup_{k \to \infty} \cF_{\varepsilon_k}[u_k] \leq \cF_0[u]$.  
Now, assume that $u \in L^2(\Omega)$ and $\cF_0[u] < \infty$, 
then $u \in H^1(\Omega)$ and 
\begin{align*}
	\cF_0[u] = \int_{\Omega} \left[ \frac{1}{2}(D_0 |\nabla u|^2 + c_0 u^2) - f_0 u \right] \;dx + \int_{\partial\Omega_1} gu\; dS.
\end{align*}
Choose $u_k = u$, for all $k = 1,2,\ldots$, then $u_k \in H^1(\Omega)$, which implies that 
\begin{align*}
	\cF_{\varepsilon_k} [u] = \int_\Omega \left[ \frac{1}{2} (D_{\varepsilon_k} |\nabla u_k|^2 + c_{\varepsilon_k} u_k^2) - f_{\varepsilon_k} u_k + g u_k |\nabla\phi_{\varepsilon_k}| \right]\; dx,
\end{align*}
for all $k = 1,2,\ldots$ 
By Lemma~\ref{diff-conv}, we have
\begin{align} 
	\lim_{k\to \infty}\int_{\Omega} D_{\varepsilon_k} |\nabla u_k|^2 \;dx &= \lim_{k\to \infty}\int_{\Omega} D_{\varepsilon_k} |\nabla u|^2\; dx = \int_{\Omega} D_0 |\nabla u|^2 \;dx, \\ 		
	\lim_{k\to \infty}\int_{\Omega} c_{\varepsilon_k}  u_k^2\; dx &= \lim_{k\to \infty}\int_{\Omega} c_{\varepsilon_k}  u^2 \;dx = \int_{\Omega} c_0 u^2 \;dx, \\
	\lim_{k\to \infty}\int_{\Omega} f_{\varepsilon_k}  u_k \;dx &= \lim_{k\to \infty}\int_{\Omega} f_{\varepsilon_k}  u \;dx = \int_{\Omega} f_0 u \;dx. 
\end{align}
Moreover, since $gu \in W^{1,1}(\Omega)$, by Lemma~\ref{diff-conv-bdry}, we have
\begin{align}
	\lim_{k\to \infty} \int_\Omega g u_k |\nabla\phi_{\varepsilon_k}| \;dx = \lim_{k\to \infty} \int_\Omega g u |\nabla\phi_{\varepsilon_k}|\; dx = \int_{\partial\Omega_1} gu \;dS.
\end{align}
Thus, we obtain
\begin{align}
	\lim_{k \to \infty} \cF_{\varepsilon_k}[u_k] 
	&= \lim_{k \to \infty} \int_\Omega \left[ \frac{1}{2} (D_{\varepsilon_k} |\nabla u_k|^2 + c_{\varepsilon_k} u_k^2) - f_{\varepsilon_k} u_k + g u_k |\nabla\phi_{\varepsilon_k}| \right]\; dx \nonumber \\
	&= \int_{\Omega} \left[ \frac{1}{2}(D_0 |\nabla u|^2 + c_0 u^2) - f_0 u \right] \;dx + \int_{\partial\Omega_1} gu\; dS \nonumber \\
	&= \cF_0[u],
\end{align}
which implies that $\{ u_k \}$ is a recovery sequence for $u$.
\end{proof}

\section{Proof of Theorem~\ref{thm-main-3}: Strong convergence of $u_\varepsilon$ in $H^1(\Omega)$}\label{sec:proof-thm-3}
In this section, we prove Theorem~\ref{thm-main-3}. 
To do that, we let $\kappa = 0$ and take an arbitrary sequence of numbers $\{ \varepsilon_k \} \subset (0,1)$ such that $\varepsilon_k \searrow 0$ as $k \to \infty$. 
For each $k = 1,2,\ldots$, let $u_{\varepsilon_k} \in H^1(\Omega)$ be the solution to the diffuse domain problem
\begin{align}
	-\nabla \cdot (D_{\varepsilon_k} \nabla u_{\varepsilon_k}) + c_{\varepsilon_k} u_{\varepsilon_k} + g|\nabla\phi_{\varepsilon_k}| &= f_{\varepsilon_k}, \quad \text{in } \Omega,  
	\\
	D_{\varepsilon_k} \nabla u_{\varepsilon_k} \cdot \boldsymbol{n}_2 & = 0, \quad \text{on } \partial\Omega.
\end{align}
Hence, $u_{\varepsilon_k}$ satisfies the weak formulation
\begin{align}\label{weak-diff-eqn}
	\int_\Omega (D_{\varepsilon_k} \nabla u_{\varepsilon_k} \cdot \nabla v  + c_{\varepsilon_k} u_{\varepsilon_k}v + g|\nabla\phi_{\varepsilon_k}|v) \; dx =  \int_\Omega f_{\varepsilon_k}v \; dx,
\end{align}
for any $v \in H^1(\Omega)$. 
Moreover, $u_{\varepsilon_k}$ is the unique minimizer of $\cF_{\varepsilon_k}$, that is,
\begin{align*}
	\cF_{\varepsilon_k}[u_{\varepsilon_k}] = \min \{ \cF_{\varepsilon_k}[u] : u \in L^2(\Omega) \} < \infty.
\end{align*}
Let $u_0$ be the solution to the two-sided problem \qref{bvp1}--\qref{bvp5} for $\kappa = 0$, that is
\begin{align*}
	u_0(x) = \begin{cases}
		u_1(x),  &\text{if } x \in \Omega_1, \\
		u_2(x),  &\text{if } x \in \Omega_2,
	\end{cases}
\end{align*}
where $u_1:\Omega_1 \to \R$ and $u_2:\Omega_2 \to \R$ satisfy
    \begin{align}
-\Delta u_1 + \gamma u_1 &= q, \quad \text{in } \Omega_1,  
    \\
-\alpha\Delta u_2 + \beta u_2 &= h, \quad \text{in } \Omega_2, 
    \\
 u_1 &= u_2, \quad \text{on } \partial\Omega_1,  
    \\
- \nabla (u_1 - \alpha u_2) \cdot \boldsymbol{n}_1 & = g, \quad \text{on } \partial\Omega_1, 
    \\
 \alpha \nabla u_2 \cdot {\boldsymbol{n}_2}  & = 0, \quad \text{on } \partial\Omega.
    \end{align}
Then, $u_0$ is the unique minimizer of $\cF_0$, that is,
\begin{align*}
	\cF_0[u_0] = \min \{ \cF_0[u] : u \in L^2(\Omega) \} < \infty.
\end{align*}
Since $\cF_\varepsilon$ $\Gamma$--converges to $\cF_0$ as $\varepsilon \to 0$, 
by the Fundamental Theorem of $\Gamma$--convergence (see, e.g., \cite{Braides2014}, Theorem~2.1), we obtain that
\begin{align*}
	u_{\varepsilon_k} \to u_0 \text{ strongly in } L^2(\Omega) \text{ as } k \to \infty, 
\end{align*}
and
\begin{align}\label{lim-Fk(uk)}
	\lim_{k \to \infty} \cF_{\varepsilon_k}[u_{\varepsilon_k}] = \cF_0[u_0].
\end{align}
By \qref{lim-Fk(uk)}, there exists a number $k_1 > 0$ such that $\cF_{\varepsilon_k}[u_{\varepsilon_k}] < \cF_0[u_0] + 1  < \infty$, for any $k \geq k_1$, 
which implies that $\{u_{\varepsilon_{k}}\}_{k=k_1}^\infty$ is bounded in $H^{1}(\Omega)$. 
By the compactness result in Theorem~\ref{compactness}, there exist a subsequence of $\{ u_{\varepsilon_k} \}_{k=k_1}^\infty$, 
which is relabeled as $\{ u_{\varepsilon_k} \}_{k=1}^\infty$, 
and a function $\tilde{u} \in H^1(\Omega)$, such that 
\begin{align*}
	u_{\varepsilon_k} &\to \tilde{u} \text{ strongly in } L^2(\Omega), \\
	u_{\varepsilon_k} &\rightharpoonup \tilde{u} \text{ weakly in } H^1(\Omega), \\
	u_{\varepsilon_k} &\to \tilde{u} \text{ a.e. in } \Omega,
\end{align*}
as $k \to \infty$. 
Hence, by the uniqueness of limits, $\tilde{u} = u_0$ a.e. in $\Omega$, 
which implies that 
\begin{align}\label{uk-weak-conv}
	u_{\varepsilon_k} \rightharpoonup u_0 \text{ weakly in } H^1(\Omega) \text{ as } k \to \infty.	
\end{align}

For each $k = 1,2,\ldots$, since $u_{\varepsilon_k} - u_0$ is an admissible test function for \qref{weak-diff-eqn}, we have
\begin{align}
	&\int_\Omega \left( D_{\varepsilon_k} \nabla u_{\varepsilon_k} \cdot (\nabla u_{\varepsilon_k} - \nabla u_0)  + c_{\varepsilon_k} u_{\varepsilon_k} (u_{\varepsilon_k} - u_0) + g |\nabla\phi_{\varepsilon_k}|(u_{\varepsilon_k} - u_0) \right) \; dx  \nonumber \\
    &=  \int_\Omega f_{\varepsilon_k} (u_{\varepsilon_k} - u_0) \; dx,
\end{align}
which implies that
\begin{align}\label{norm-est}
	&\int_\Omega (D_{\varepsilon_k} |\nabla u_{\varepsilon_k} - \nabla u_0|^2 + c_{\varepsilon_k} |u_{\varepsilon_k} - u_0|^2) \; dx \nonumber \\
	&= \int_\Omega (f_{\varepsilon_k} - g|\nabla\phi_{\varepsilon_k}|) (u_{\varepsilon_k} - u_0) \; dx 
	+ \int_\Omega D_{\varepsilon_k}(|\nabla u_0|^2 - \nabla u_{\varepsilon_k} \cdot \nabla u_0  )\; dx  \nonumber \\
    &+ \int_\Omega c_{\varepsilon_k} (u_0^2 -  u_{\varepsilon_k}u_0) \; dx.
\end{align}

For the first term on the right hand side of \qref{norm-est}, since $u_0, g \in H^1(\Omega)$, applying Lemmas~\ref{diff-conv-bdry} and \ref{diff-conv}, we get
\begin{align}
	\int_\Omega (f_{\varepsilon_k} - g|\nabla\phi_{\varepsilon_k}|) u_0 \; dx \to \int_\Omega f_0 u_0 \; dx - \int_{\partial\Omega_1} gu_0 \; dS,
\end{align}
as $k \to \infty$. 
On the other hand, by applying arguments analogous to those used for \qref{liminf-4} and \qref{liminf-3} in the proof of Theorem~\ref{compactness}, we obtain
\begin{align}
	\int_\Omega (f_{\varepsilon_k} - g|\nabla\phi_{\varepsilon_k}|) u_{\varepsilon_k} \; dx \to \int_\Omega f_0 u_0 \; dx - \int_{\partial\Omega_1} gu_0 \; dS,
\end{align}
as $k \to \infty$. 
Therefore,
\begin{align}\label{norm-est-1}
	\int_\Omega (f_{\varepsilon_k} - g|\nabla\phi_{\varepsilon_k}|) (u_{\varepsilon_k} - u_0) \; dx \to 0  \quad \text{as } k \to \infty.
\end{align}

Now we examine the second term on the right hand side of \qref{norm-est}. 
Since $\nabla u_0 \in L^2(\Omega)$, applying Lemma~\ref{diff-conv}, we get
\begin{align}\label{norm-est-2a}
	\int_\Omega D_{\varepsilon_k}|\nabla u_0|^2 \; dx \to \int_\Omega D_0|\nabla u_0|^2 \; dx, 
\end{align}
as $k \to \infty$. 
Furthermore, since $|D_{\varepsilon_k} \nabla u_0 - D_0 \nabla u_0|^2 \to 0$ a.e. in $\Omega$, 
$|D_{\varepsilon_k} \nabla u_0 - D_0 \nabla u_0|^2 \leq  (2\max\{ \alpha, 1 \} |\nabla u_0|)^2$, and $\nabla u_0 \in L^2(\Omega)$, 
applying the Dominated Convergence Theorem, we obtain
\begin{align}
    \|D_{\varepsilon_k} \nabla u_0 - D_0 \nabla u_0\|_{L^2(\Omega)}^2 = \int_\Omega |D_{\varepsilon_k} \nabla u_0 - D_0 \nabla u_0|^2 \; dx \to 0,
\end{align}
as $k \to \infty$, 
which implies that $D_{\varepsilon_k} \nabla u_0 \to D_0 \nabla u_0$ strongly in $L^{2}(\Omega)$ as $k \to \infty$. 
This convergence of $D_{\varepsilon_k} \nabla u_0$ and the weak convergence $\nabla u_{\varepsilon_{k}} \rightharpoonup \nabla u_0$ in $L^{2}(\Omega)$ together imply that 
\begin{align}\label{norm-est-2b}
	\int_\Omega   D_{\varepsilon_k}\nabla u_{\varepsilon_k} \cdot \nabla u_0  \; dx 
    &=\int_\Omega   \nabla u_{\varepsilon_k} \cdot D_{\varepsilon_k} \nabla u_0  \; dx  \nonumber \\
    &\to  \int_{\Omega}   \nabla u_0 \cdot D_0 \nabla u_0 \; dx
    = \int_\Omega D_0 |\nabla u_0|^2 \; dx,
\end{align}
as $k \to \infty$. 
Combining \qref{norm-est-2a} and \qref{norm-est-2b}, we get
\begin{align}\label{norm-est-2}
	\int_\Omega D_{\varepsilon_k}(|\nabla u_0|^2 -  \nabla u_{\varepsilon_k} \cdot \nabla u_0  ) \; dx \to 0  \quad \text{as } k \to \infty.
\end{align}
Applying a similar argument for the third term on the right hand side of \qref{norm-est}, we also obtain
\begin{align}\label{norm-est-3}
	\int_\Omega c_{\varepsilon_k}(u_0^2 - u_{\varepsilon_k} u_0)\; dx \to 0  \quad \text{as } k \to \infty.
\end{align}
Using \qref{norm-est-1}, \qref{norm-est-2} and \qref{norm-est-3}, 
taking limits on both sides of \qref{norm-est} as $k \to \infty$, 
and recalling that $\omega = \min \{ \alpha, \beta, \gamma, 1 \} > 0$, we have
\begin{align*}
	\| u_{\varepsilon_k} - u_0 \|_{H^1(\Omega)}^2 \leq 
	\frac{1}{\omega}\int_\Omega (D_{\varepsilon_k} |\nabla u_{\varepsilon_k} - \nabla u_0|^2 + c_{\varepsilon_k} |u_{\varepsilon_k} - u_0|^2) \; dx \to 0  \quad \text{as } k \to \infty, 
\end{align*}
which implies that $u_{\varepsilon_k}$ converges strongly to $u_0$ in $H^1(\Omega)$ as $k \to \infty$. 
Note that this sequence is merely a subsequence of the original sequence $\{u_{\varepsilon_k}\}$ chosen at the beginning of this section.

\begin{rema}
    If we impose the Robin boundary condition 
    \begin{align*}
        -\nabla (u_1 - \alpha u_2) \cdot \text{\textbf{n}}_1  = \kappa u_1 + g, \quad \text{on } \partial\Omega_1,
    \end{align*}    
    with $\kappa > 0$, 
    then establishing the liminf inequality poses a significant challenge due to the presence of the term 
    \begin{align*}
        \int_\Omega  \frac{1}{2}\kappa u_\varepsilon^2 |\nabla\phi_\varepsilon| dx,
	\end{align*}    
    which must be shown to converge, up to a subsequence, to
    \begin{align*}
        \int_{\partial\Omega_1} \frac{1}{2}\kappa u^2 dS,
	\end{align*}    
    as $\varepsilon \to 0$. 
    It is not clear how to show this convergence in general. 
    However, leveraging the Sobolev Embedding Theorem in 1D, we will demonstrate in Section~\ref{1D-case} that, this convergence does hold in 1D. 
    This, in turn, implies that the $\Gamma$--convergence and strong $H^1(\Omega)$--convergence results hold for the Robin boundary condition case in 1D.
\end{rema}

\section{$\Gamma$-convergence and Strong $H^1(\Omega)$--convergence in 1D}\label{1D-case}
In this section, we consider the two-sided problem with transmission-type Robin boundary conditions in 1D. 
We consider the two-sided problem \qref{bvp1}--\qref{bvp5} over the following domains in 1D:
\begin{align*}
    \Omega = (a, b), \quad \Omega_1 = (a_1, b_1), \quad \Omega_2 = (a, a_1) \cup (b_1, b), 
\end{align*}
where $a < a_1 < b_1 < b$. 
Additionally, we extend the definitions of $\cE_0$ and $\cE_\varepsilon$ from \qref{E(u)} and \qref{E-eps}, respectively, to all $u \in L^2(\Omega)$, by defining:
\begin{align}\label{F(u)-1D}
	\cF_0 [u] = 
	\begin{cases} 
        \cE_0 [u]  , &  u \in H^1(\Omega), 
		\\
		\infty, & u \in L^2(\Omega) \setminus H^1(\Omega), 
	\end{cases}
\end{align}
and
\begin{align}\label{F-eps-1D}
	\cF_\varepsilon [u] = 
	\begin{cases} 
        \cE_\varepsilon [u]  , &  u \in H^1(\Omega), 
		\\
		\infty, & u \in L^2(\Omega) \setminus H^1(\Omega). 
	\end{cases}
\end{align}
We will show that, in 1D, the $\Gamma$--convergence result (similar to that in Theorems~\ref{thm-main-2}) and the strong $H^1(\Omega)$--convergence result (similar to that in Theorems~\ref{thm-main-3}) hold for any $\kappa \geq 0$.

\subsection{$\Gamma$--convergence of the Energy Functional in 1D} \label{sec:proof-thm-2-1D}
We will show that Theorems~\ref{compactness}, \ref{liminf-ineq} and \ref{limsup-ineq} hold in 1D, for any $\kappa \geq 0$. 
We only demonstrate the estimates and convergences related to the term $\int_\Omega  \frac{1}{2}\kappa u_\varepsilon^2 |\nabla\phi_\varepsilon| dx$. 
The rest of the proofs will be similar to those in Section~\ref{sec:proof-thm-2}.

For the compactness result,  
using the same argument as in Section~\ref{sec:compactness},
combining with $\int_\Omega \frac{1}{2} \kappa u_k^2 |\nabla\phi_{\varepsilon_k}| dx \geq 0$, 
we obtain the same compactness result as in Theorem~\ref{compactness}. 

For the limsup inequality and recovery sequence, with the choice $u_k = u$, for all $k = 1,2,\ldots$, 
applying Lemma~\ref{diff-conv-bdry} for $u^2 \in W^{1,1}(\Omega)$, we get
\begin{align*}
	\lim_{k\to \infty} \int_\Omega \frac{1}{2} \kappa u_k^2 |\nabla\phi_{\varepsilon_k}| \;dx = \lim_{k\to \infty} \int_\Omega \frac{1}{2} \kappa u^2 |\nabla\phi_{\varepsilon_k}|\; dx = \int_{\partial\Omega_1} \frac{1}{2} \kappa u^2 \;dS.
\end{align*}
Then, using the same argument as in Section~\ref{sec:limsup-ineq}, $\{ u_k \}$ is a recovery sequence for $u$.

Now we handle the liminf inequality. 
Let $\{ \varepsilon_k \} \subset (0,1)$ be a sequence of numbers such that $\varepsilon_k \searrow 0$ as $k \to \infty$. 
Making assumptions similar to (A1)--(A4) in Section~\ref{sec:liminf-ineq}, but with $\cF_0$ and $\cF_\varepsilon$ defined by \qref{F(u)-1D} and \qref{F-eps-1D}, respectively, 
we only need to show that
\begin{align*}
	\lim_{k\to \infty} \int_\Omega \frac{1}{2} \kappa u_k^2 |\nabla\phi_{\varepsilon_k}| \;dx = \int_{\partial\Omega_1} \frac{1}{2} \kappa u^2 \;dS,
\end{align*}
up to a subsequence. 
The rest of the argument is similar to that in the proof of Theorem~\ref{liminf-ineq}. 

By Sobolev Embedding Theorem for $n=1$ (see, e.g., \cite{Evans2}, Section~5.6.3, Theorem~6), there exists a constant $K > 0$, depending only on $\Omega$, such that, for any $w \in H^1(\Omega)$,
\begin{align*}
    \| w \|_{C^{0, 1/2}(\overline{\Omega})} \leq K \| w\|_{H^1(\Omega)}.
\end{align*}
Combining the above estimate with \qref{bnd-uk-phi-k}, we have $u_k \in C^{0, 1/2}(\overline{\Omega})$ and
\begin{align*}
	\| u_k \|_{C^{0, 1/2}(\overline{\Omega})}^2 
    &\leq K^2 \| u_k \|_{H^1(\Omega)}^2 \\
    &\leq \frac{4K^2}{\omega} \left( L + 1 + \frac{(C_0^2 + 1) M_1}{\omega} \right) := M_2,
\end{align*}
for any $k = 1,2,\ldots$. 
This implies that, for each $k = 1,2,\ldots$, 
\begin{align*}
	\| u_k \|_{L^\infty(\Omega)}    \leq \sqrt{M_2},
\end{align*}
and
\begin{align*}
	| u_k(x) - u_k(y) |   \leq \sqrt{M_2} |x-y|^{1/2}, \quad \text{ for any } x, y \in \Omega.
\end{align*}
Therefore, $\{ u_k \}$ is uniformly bounded and equicontinuous. 
Hence, by the Arzel\`{a}-Ascoli Theorem (see, e.g., \cite{Evans2}, Appendix C.8), 
there exist a subsequence of $\{ u_k \}$ (not relabeled)
and a function $\bar{u} \in C(\overline{\Omega})$ such that $u_k$ converges uniformly to $\bar{u}$, as $k \to \infty$. 
Since  $u_k \to u$ a.e. in $\Omega$ as $k \to \infty$ (by Assumption (A4)), then $\bar{u} = u$ a.e. in $\Omega$. 
Hence, $u \in L^\infty (\Omega)$ and $u_k \to u$ strongly in $L^\infty (\Omega)$ as $k \to \infty$. 

Applying Lemma~\ref{diff-conv-bdry} for $u^2 \in W^{1,1}(\Omega)$, we get
\begin{align*}
	\lim_{k\to \infty} \int_\Omega \frac{1}{2} \kappa u^2 |\nabla\phi_{\varepsilon_k}|\; dx = \int_{\partial\Omega_1} \frac{1}{2} \kappa u^2 \;dS.
\end{align*}
Then, using $\int_\Omega |\nabla\phi_{\varepsilon_k}| dx = 1$, we obtain
\begin{align*}
	&\left| \int_\Omega \frac{1}{2} \kappa u_k^2 |\nabla\phi_{\varepsilon_k}|\; dx - \int_{\partial\Omega_1} \frac{1}{2} \kappa u^2 \;dS \right| \\
    &\leq \left|  \int_\Omega \frac{1}{2} \kappa (u_k^2 - u^2) |\nabla\phi_{\varepsilon_k}|\; dx \right|
    + \left| \int_\Omega \frac{1}{2} \kappa u^2 |\nabla\phi_{\varepsilon_k}|\; dx - \int_{\partial\Omega_1} \frac{1}{2} \kappa u^2 \;dS \right|    \\
    &\leq \frac{1}{2} \kappa \| u_k^2 - u^2 \|_{L^\infty (\Omega)} \int_\Omega |\nabla\phi_{\varepsilon_k}|\; dx  
    + \left| \int_\Omega \frac{1}{2} \kappa u^2 |\nabla\phi_{\varepsilon_k}|\; dx - \int_{\partial\Omega_1} \frac{1}{2} \kappa u^2 \;dS \right|    \\
    &\leq \frac{1}{2} \kappa \| u_k^2 - u^2 \|_{L^\infty (\Omega)} 
    + \left| \int_\Omega \frac{1}{2} \kappa u^2 |\nabla\phi_{\varepsilon_k}|\; dx - \int_{\partial\Omega_1} \frac{1}{2} \kappa u^2 \;dS \right|   
    \to 0 \quad \text{as } k \to \infty,
\end{align*}
which implies that
\begin{align*}
	\lim_{k\to \infty} \int_\Omega \frac{1}{2} \kappa u_k^2 |\nabla\phi_{\varepsilon_k}| \;dx = \int_{\partial\Omega_1} \frac{1}{2} \kappa u^2 \;dS.
\end{align*}
This completes the proof of the $\Gamma$--convergence result for the Robin boundary condition case in 1D.

\subsection{Strong $H^1(\Omega)$--convergence of the Approximation Solution in 1D}
We will show that Theorem~\ref{thm-main-3} holds in 1D, for any $\kappa \geq 0$. 
Choose an arbitrary sequence of numbers $\{ \varepsilon_k \} \subset (0,1)$ such that $\varepsilon_k \searrow 0$ as $k \to \infty$. 
For each $k = 1,2,\ldots$, let $u_{\varepsilon_k} \in H^1(\Omega)$ satisfy the weak formulation
\begin{align}\label{weak-diff-eqn-1D}
	\int_\Omega (D_{\varepsilon_k} \nabla u_{\varepsilon_k} \cdot \nabla v  + c_{\varepsilon_k} u_{\varepsilon_k}v + (\kappa u_{\varepsilon_k} + g)|\nabla\phi_{\varepsilon_k}|v) \; dx =  \int_\Omega f_{\varepsilon_k}v \; dx,
\end{align}
for any $v \in H^1(\Omega)$, 
then $u_{\varepsilon_k}$ is the unique minimizer of $\cF_{\varepsilon_k}$, defined by \qref{F-eps-1D}.
Let $u_0$ be the solution to the two-sided problem \qref{bvp1}--\qref{bvp5},
then $u_0$ is the unique minimizer of $\cF_0$, defined by \qref{F(u)-1D}. 
Using an argument similar to that in Section~\ref{sec:proof-thm-3}, there exists a subsequence of $\{ u_{\varepsilon_k} \}$ (not relabeled) that is bounded in $H^1(\Omega)$, such that
\begin{align*}
	u_{\varepsilon_k} &\to u_0 \text{ strongly in } L^2(\Omega), \\
	u_{\varepsilon_k} &\rightharpoonup u_0 \text{ weakly in } H^1(\Omega), \\
	u_{\varepsilon_k} &\to u_0 \text{ a.e. in } \Omega,
\end{align*}
as $k \to \infty$. 
Then, an argument similar to that in Section~\ref{sec:proof-thm-2-1D} gives that $\{ u_{\varepsilon_k} \}$ is bounded in $L^\infty(\Omega)$, that is,
\begin{align*}
	\| u_{\varepsilon_k} \|_{L^\infty(\Omega)}    \leq M_3, \quad \text{for all } k = 1,2,\ldots,
\end{align*}
for some $M_3 > 0$ independent of $k$, 
and $u_{\varepsilon_k} \to u$ strongly in $L^\infty (\Omega)$ as $k \to \infty$. 

Similar to the proof in Section~\ref{sec:proof-thm-3}, we use $u_{\varepsilon_k} - u_0$ as a test function for \qref{weak-diff-eqn-1D}, which implies that
\begin{align*}
	&\int_\Omega \left( D_{\varepsilon_k} \nabla u_{\varepsilon_k} \cdot (\nabla u_{\varepsilon_k} - \nabla u_0)  + c_{\varepsilon_k} u_{\varepsilon_k} (u_{\varepsilon_k} - u_0) + (\kappa u_{\varepsilon_k} + g) |\nabla\phi_{\varepsilon_k}|(u_{\varepsilon_k} - u_0) \right) \; dx  \nonumber \\
    &=  \int_\Omega f_{\varepsilon_k} (u_{\varepsilon_k} - u_0) \; dx.
\end{align*}
The equation above is equivalent to
\begin{align}\label{norm-est-1D}
	&\int_\Omega (D_{\varepsilon_k} |\nabla u_{\varepsilon_k} - \nabla u_0|^2 + c_{\varepsilon_k} |u_{\varepsilon_k} - u_0|^2) \; dx  \nonumber \\
	&= \int_\Omega (f_{\varepsilon_k} - g|\nabla\phi_{\varepsilon_k}|) (u_{\varepsilon_k} - u_0) \; dx 
	+ \int_\Omega D_{\varepsilon_k}(|\nabla u_0|^2 - \nabla u_{\varepsilon_k} \cdot \nabla u_0  )\; dx  \nonumber \\
    &+ \int_\Omega c_{\varepsilon_k} (u_0^2 -  u_{\varepsilon_k}u_0) \; dx + \int_\Omega \kappa u_{\varepsilon_k} (u_0 - u_{\varepsilon_k}) |\nabla\phi_{\varepsilon_k}| \; dx.
\end{align}
We only need to show that the last integral on the right-hand side of \qref{norm-est-1D} approaches 0 as $k \to \infty$. 
The rest of the proof is similar to that in Section~\ref{sec:proof-thm-3}. 

Since $\int_\Omega |\nabla\phi_{\varepsilon_k}| dx = 1$ 
and $u_{\varepsilon_k} \to u$ strongly in $L^\infty (\Omega)$ as $k \to \infty$, we have
\begin{align*}
    \left| \int_\Omega \kappa u_{\varepsilon_k} (u_0 - u_{\varepsilon_k}) |\nabla\phi_{\varepsilon_k}| \; dx \right| 
    &\leq \kappa \|u_{\varepsilon_k}\|_{L^\infty(\Omega)} \|u_0 - u_{\varepsilon_k}\|_{L^\infty(\Omega)} \int_\Omega |\nabla\phi_{\varepsilon_k}| \; dx \nonumber \\
    &\leq \kappa M_3 \|u_0 - u_{\varepsilon_k}\|_{L^\infty(\Omega)} 
    \to 0 \quad \text{as } k \to \infty, 
\end{align*}
which implies that
\begin{align*}
    \lim_{k \to \infty} \int_\Omega \kappa u_{\varepsilon_k} (u_0 - u_{\varepsilon_k}) |\nabla\phi_{\varepsilon_k}| \; dx = 0. 
\end{align*}
This completes the proof of the strong $H^1(\Omega)$--convergence result for the Robin boundary condition case in 1D.

    \section{Discussion}
    \label{sec:discuss}
In this work, we established the $\Gamma$--convergence for the diffuse domain energy functional and the strong $H^1(\Omega)$--convergence for the diffuse domain approximation solution.  
The two-sided problem is both theoretically interesting and practically relevant, as many physical processes exhibit different parameters on either side of an interface.
 
However, as we discussed in \cite{LuongDDM2025-1}, the one-sided problem is arguably more intriguing and serves as the primary motivation for diffuse domain methods. 
Some open questions arise: Can the $\Gamma$--convergence and $H^1(\Omega)$--convergence analyses be extended to cases where $\alpha$ depends on $m$, specifically,  $\alpha = \varepsilon^m$, for some constant $m > 0$? 
This change breaks the established $\Gamma$--convergence result, as the associated sharp interface energy functional $\cE_0$ over the whole cuboidal domain $\Omega$, for the one-sided problem, is still unknown. 
In particular, it seems, we need to use a limiting energy in $\Omega_2$ that is zero. 
Moreover, what happens to the solution $u_2$ in the exterior domain, $\Omega_2$, in this case?
We refer readers to \cite{LuongDDM2025-1} for further discussion and numerical experiments.

    \section*{Acknowledgments}

TM gratefully acknowledges funding from the US National Science Foundation via grant number NSF-DMS 2206252. TL is thankful for partial support from the Department of Mathematics at the University of Tennessee. SMW and MHW acknowledge generous support from the US National Science Foundation through grants NSF-DMS 2012634 and NSF-DMS 2309547. The authors thank John Lowengrub for helpful discussions about this topic.

    \bibliographystyle{plain}
    \bibliography{Luong-bib}

    \end{document}